\documentclass{amsart}

\usepackage{amsmath,amsfonts,amssymb,color,hyperref, enumerate,tabularx}

\usepackage{amsmath, amsfonts,amsthm,amssymb,amscd, verbatim,graphicx,color,multirow,tikz,tikz-cd,booktabs, caption, mathdots,bm,chngcntr,adjustbox,longtable%,enumitem
}
\usepackage{tikz-cd}
\usetikzlibrary{positioning}
\newtheorem{theorem}{Theorem}[section]
\newtheorem{lemma}[theorem]{Lemma}
\newtheorem{proposition}[theorem]{Proposition}
\newtheorem{notation}[theorem]{Notation}

\newtheorem{remark}[theorem]{Remark}
\newtheorem{conjecture}[theorem]{Conjecture}

\newcommand{\Sym}{\mathop{\mathrm{Sym}}}

\begin{document}

\title[Cliques in derangement graph]{Kronecker classes and cliques in  derangement graphs}

\author[M.~Cazzola]{Marina Cazzola}
\address{Dipartimento di Matematica e Applicazioni, University of Milano-Bicocca, Via Cozzi 55, 20125 Milano, Italy} 
\email{marina.cazzola@unimib.it}
\author[L.~Gogniat]{Louis Gogniat}
\address{Institute of Mathematics, EPFL, Station 8
CH-1015 Lausanne, Switzerland} 
\email{louis.gogniat@epfl.ch}

\author[P.~Spiga]{Pablo Spiga}
\address{Dipartimento di Matematica e Applicazioni, University of Milano-Bicocca, Via Cozzi 55, 20125 Milano, Italy} 
\email{pablo.spiga@unimib.it}
\begin{abstract}
Given a permutation group $G$, the derangement graph of $G$ is defined with vertex set $G$, where two elements $x$ and $y$ are adjacent if and only if $xy^{-1}$ is a derangement. We establish that, if $G$ is transitive with degree exceeding 30, then the derangement graph of $G$ contains a complete subgraph with four vertices.

As a consequence, if $G$ is a normal subgroup of $A$ such that $|A : G| = 3$, and if $U$ is a subgroup of $G$ satisfying $G = \bigcup_{a \in A} U^a$, then $|G : U| \leq 10$. This result provides support for a conjecture by Neumann and Praeger concerning Kronecker classes.
\keywords{derangement graph, independent sets, Erd\H{o}s-Ko-Rado Theorem, Symmetric Group, multipartite graphs}
\end{abstract}

\subjclass[2010]{Primary 05C35; Secondary 05C69, 20B05}

\maketitle

\section{Introduction}\label{intro}
The Erd\H{o}s-Ko-Rado theorem~\cite{erdos1961intersection} is a landmark result in extremal combinatorics. It states that for integers $n$ and $k$ with $1 \leq 2k < n$, any family $\mathcal{F}$ of $k$-subsets of $\{1, \ldots, n\}$ in which every pair of sets intersects has cardinality at most $\binom{n-1}{k-1}$. Moreover, this upper bound is achieved if and only if there exists an element $x \in \{1, \ldots, n\}$ contained in every member of $\mathcal{F}$.

Numerous extensions of the Erd\H{o}s-Ko-Rado theorem have been explored across various combinatorial structures. This paper focuses on analogues within permutation groups. For a finite permutation group $G$ acting on a set $\Omega$, a subset $\mathcal{F} \subseteq G$ is \textit{\textbf{intersecting}} if, for every $g, h \in \mathcal{F}$, the permutation $gh^{-1}$ fixes at least one point in $\Omega$. This condition generalizes the definition of intersecting families in the classical theorem, as $gh^{-1}$ fixing a point corresponds to agreement in at least one coordinate when $g$ and $h$ are expressed as tuples.

For each $\omega \in \Omega$, the point stabilizer $G_\omega$ is an intersecting subset of $G$. More broadly, any coset of a point stabilizer is also intersecting. The analogue of the Erd\H{o}s-Ko-Rado theorem for $G = \mathrm{Sym}(\Omega)$ was independently proven by Cameron-Ku~\cite{6} and Larose-Malvenuto~\cite{10}, demonstrating that intersecting sets in $\mathrm{Sym}(\Omega)$ have cardinality at most $(|\Omega| - 1)!$, with equality achieved only by cosets of point stabilizers. However, for general permutation groups, these bounds do not hold universally: point stabilizers are not always of maximum size among intersecting sets,\footnote{For example, in the action of $\mathrm{Alt}(5)$ on the 10 subsets of size 2 from $\{1, 2, 3, 4, 5\}$, the intersecting set $\mathrm{Alt}(4)$ has size 12, whereas the point stabilizer has size 6.} and even when they are, not all maximal intersecting sets are cosets of point stabilizers.\footnote{For instance, in $\mathrm{PGL}_d(q)$ acting $2$-transitively on the points of the projective space $\mathrm{PG}_{d-1}(q)$, the largest intersecting sets are either cosets of point stabilizers or of hyperplane stabilizers, as shown in~\cite{spiga}.} These challenges make the study of maximal intersecting sets in arbitrary permutation groups both intriguing and complex. 

To quantify the deviation of a group $G$ from satisfying an Erd\H{o}s-Ko-Rado-like property, Li, Song, and Pantagi introduced the concept of \textit{\textbf{intersection density}}. For $\omega \in \Omega$, let $G_\omega$ denote the largest point stabilizer in $G$.\footnote{In transitive groups, all point stabilizers have the same cardinality.} The \textit{\textbf{intersection density}} of an intersecting subset $\mathcal{F} \subseteq G$ is defined as 
$$\rho(\mathcal{F}) = \frac{|\mathcal{F}|}{|G_\omega|},$$ 
and the \textit{\textbf{intersection density}} of $G$ is 
$$\rho(G) = \max \{\rho(\mathcal{F}) \mid \mathcal{F} \subseteq G, \mathcal{F} \text{ is intersecting} \}.$$ 
This invariant, introduced in~\cite{li2020ekr}, measures the extent to which $G$ satisfies an Erd\H{o}s-Ko-Rado-like theorem.

Let $\mathcal{D}$ denote the set of all \textit{\textbf{derangements}} of $G$, where a derangement is a permutation without fixed points. The \textit{\textbf{derangement graph}} $\Gamma_G$ of $G$ has vertex set $G$, with edges connecting $x, y \in G$ if $xy^{-1} \in \mathcal{D}$. Equivalently, $\Gamma_G$ is the Cayley graph of $G$ with connection set $\mathcal{D}$. In this framework, intersecting families in $G$ correspond to independent sets (or cocliques) in $\Gamma_G$. Let $\omega(\Gamma_G)$ and $\alpha(\Gamma_G)$ denote the sizes of the largest cliques and independent sets, respectively. The clique-coclique bound~\cite[Theorem~2.1.1]{GodsilMeagher} 
\begin{equation}\label{clique-coclique} \alpha(\Gamma_G)\omega(\Gamma_G) \leq |G| 
\end{equation} 
relates these parameters and provides insights into intersection density:
\begin{equation}\label{clique-coclique2} \rho(G) \leq \frac{|\Omega|}{\omega(\Gamma_G)}.
\end{equation}

Theorem~1.5 of~\cite{KRS} establishes that for transitive $G$ with $|\Omega| \geq 3$, $\Gamma_G$ contains a triangle, implying $\rho(G) \leq |\Omega|/3$. Question~6.1 of~\cite{KRS} asks whether a function $f : \mathbb{N} \to \mathbb{N}$ exists such that if $G$ is transitive of degree $n$ and $\Gamma_G$ lacks a $k$-clique, then $n \leq f(k)$.\footnote{A related question on weak normal coverings of groups appears in~\cite{BSW}.} When $k = 2$, Jordan's theorem gives $n \leq 1$, and for $k = 3$, Theorem~1.5 of~\cite{KRS} gives $n \leq 2$. This paper contributes further evidence toward this question.

\begin{theorem}\label{thrm:main}
Let $G$ be a finite transitive permutation group on $\Omega$. Then either the derangement graph of $G$ in its action on $\Omega$ has a clique of size at least $4$, or one of the following holds
\begin{enumerate}
\item$\label{thrm:maineq1}$ $|\Omega|\le 3$ and $\mathrm{Alt}(\Omega)\le G\le\mathrm{Sym}(\Omega)$,
\item$\label{thrm:maineq2}$ $|\Omega|=6$ and $G\cong\mathrm{Alt}(4)$,
\item$\label{thrm:maineq3}$ $|\Omega|=18$, $|G|=324$ and $G$ is the group $(18,142)$ in the library of transitive groups in \texttt{magma},
\item$\label{thrm:maineq4}$ $|\Omega|=30$, and $G$ is the group $(30,126)$ or $(30,233)$ in the library of transitive groups in \texttt{magma}.
\end{enumerate}

The derangement graph of the groups in parts~\eqref{thrm:maineq1}--~\eqref{thrm:maineq4} do not admit a clique of cardinality $4$.
\end{theorem}

Various open conjectures regarding Kronecker classes of field extensions can be reformulated as questions about covering groups by conjugates of subgroups. The interplay between these two subjects is detailed in~\cite{Klingen98} and summarized in~\cite[Section1.7]{BSW}; we also refer the reader to the papers~\cite{GMP,Jehne77,Klingen78,Praeger88,Praeger94}. For applications to Kronecker classes of field extensions, the most important open problem is the following conjecture of Neumann and Praeger~\cite[Conjecture~4.3]{Praeger94} (see also the Kourovka Notebook \cite[11.71]{notebook}).

\begin{conjecture}[{Neumann-Praeger}]\label{conj}{\rm
Let $A$ be a finite group with a normal  subgroup $H$. A subgroup $U$ of $H$ is 
called an $A$-covering subgroup of $H$ if $H=\bigcup_{a\in A}U^a$. 

Is there a function $f : \mathbb{N} \to \mathbb{N}$
such that whenever $U < H \le A$, where $A$ is a finite group, $H$ is a normal subgroup
of $A$ of index $n$, and $U$ is an $A$-covering subgroup of $H$, the index $|H : U | \le f (n)$?
}
\end{conjecture}
This conjecture is obvious when $n=1$, and it is one of the main results in~\cite{Saxl} when $n=2$. Theorem~\ref{thrm:main} proves Conjecture~\ref{conj} when $n=3$.\footnote{In fact, in~\cite[page~934]{FPS}, the authors have shown that, if Question~6.1 in~\cite{KRS} has a positive answer, then the Praeger-Neumann conjecture~\cite[11.71]{notebook} holds true. It is unclear if the converse is also true.} 
\begin{proof}[Proof of Conjecture~$\ref{conj}$ when $n=3$] Let $\Omega$ be the set of right cosets of $U$ in $A$ and let us consider the action of $A$ on $\Omega$. As $H=\bigcup_{a\in A}U^a$ and as $U$ is the stabilizer of a point in $\Omega$, we deduce that  $H$ contains no derangement.

Let $C$ be a clique for the derangement graph of $A$ in its action on $\Omega$. Since $xy^{-1}$ is a derangement for any two distinct elements  $x$ and $y$ of $C$, it follows that $x$ and $y$ belong to distinct $H$-cosets in $A$. Therefore, $|C|\le |A:H|=3$. Thus, by Theorem~\ref{thrm:main}, we deduce $|A:H|\le 30$ and hence $|H:U|\le 10$.
\end{proof}

\section{Preliminaries}\label{sec:preliminaries}

\begin{lemma}\label{centralizes}
Let $A$ and $B$ be transitive subgroups of $\mathrm{Sym}(\Omega)$. If $A$ and $B$ commute, then $B={\bf C}_{\mathrm{Sym}(\Omega)}(A)$ and $A={\bf C}_{\mathrm{Sym}(\Omega)}(B)$.
\end{lemma}
\begin{proof}
See~\cite[page~27,~1.5]{Peter}.
\end{proof}

\begin{notation}{\rm  
For future reference, we provide a summary of key definitions, some of which were previously introduced in the introductory section.

Let $G$ be an abstract group and let $X$ and $Y$ be subgroups of $G$. We denote by $$[X,Y]:=\langle [x,y]\mid x\in X,y\in Y\rangle$$
the subgroup of $G$ generated by the commutators $[x,y]=x^{-1}y^{-1}xy$, with $x\in X$ and $y\in Y$.

Let $\Sigma$ be a system of imprimitivity for the action of $G$ on $\Omega$ and let $\Delta\in \Sigma$. 
We let $$G_{\{\Delta\}}=\{g\in G\mid \Delta^g=\Delta\}$$ be the setwise stabilizer of $\Delta$ in $G$ and we let $$G_{(\Delta)}=\{g\in G\mid \delta^g=\delta, \,\forall \delta\in \Delta\}$$ be the pointwise stabilizer of $\Delta$ in $G$. Using parentheses and braces makes the notation cumbersome, but it has the advantage of avoiding some misunderstandings.

Consider a group $G$ acting on a set $\Omega$. A \textit{\textbf{derangement}} of $G$ is an element $g$ lacking fixed points, that is, $\{\omega\in \Omega\mid \omega^g=\omega\}=\emptyset$. Let  $\mathcal{D}$ denote the set of all derangements within $G$. The \textit{\textbf{derangement graph}} $\Gamma_{G,\Omega}$ of $G$ in its action on $\Omega$ is formed by vertices representing elements of $G$, and edges connecting pairs $(h, g) \in G \times G$ where $gh^{-1} \in \mathcal{D}$. In essence, $\Gamma_{G,\Omega}$ constitutes the \textit{\textbf{Cayley graph}} of $G$ with connection set $\mathcal{D}$. Consequently, an intersecting family of $G$ corresponds to an independent set or coclique in $\Gamma_{G,\Omega}$, and vice versa.

Traditionally, the definition of the derangement graph confines $G$ to act faithfully on $\Omega$, implying that $G$ is a permutation group on $\Omega$. However, in our context, it proves more meaningful to embrace the broader spectrum of group actions. This flexibility is essential, particularly in the course of our work involving inductive arguments, where we often encounter the action of a finite transitive group $G$ on a system of imprimitivity. In such scenarios, there arises a frequent need to lift cliques from the derangement graph of $G$ in its action on the system of imprimitivity to cliques in the derangement graph of $G$ in its action on $\Omega$.

Let $\Gamma$ be a graph and let $G\le\mathrm{Aut}(\Gamma)$ be a group of automorphisms of $\Gamma$. The graph $\Gamma$ is said to be $G$-\textbf{\textit{vertex-transitive}} if $G$ acts transitively on the vertices of $\Gamma$. Moreover, $\Gamma$ is $G$-\textbf{\textit{edge-transitive}} if $G$ acts transitively on the edges of $\Gamma$. These definitions are relevant in Proposition~\ref{prop:technique2}.}

\begin{lemma}\label{centralizers2}
Let $A$ and $B$ be two commuting regular subgroups of $\mathrm{Sym}(\Delta)$. Then either there exist $a\in A\setminus\{1\}$ and $b\in B\setminus \{1\}$ such that $ab$ is a derangement, or $|\Delta|\le 2$. 
\end{lemma}
\begin{proof}
From the basic theory of permutation groups, we can identify $\Delta$ with the elements of an abstract finite group $X$, we can identify $A$ with the right regular representation of $X$ and we can identify $B$ with the left regular representation of $X$. Thus, $A=\{\rho_x:x\in X\}$, where $\omega\mapsto \omega^{\rho_x}=\omega x$, for every $\omega\in X$. Similarly, $B=\{\lambda_x:x\in X\}$, where $\omega\mapsto \omega^{\lambda_x}=x^{-1}\omega $, for every $\omega\in X$.

Assume $|X|=|\Delta|\ne 1$. Let $x,y\in X\setminus\{1\}$ and consider $\rho_x\lambda_y$. If $\rho_x\lambda_y$ is not a derangement on $X$, then there exists $\omega\in X$ with $$\omega=\omega^{\rho_x\lambda_y}=y^{-1}\omega x,$$
that is, $\omega^{-1} x\omega=y$. In particular, $x$ and $y$ belong to the same $X$-conjugacy class. This shows that as long as $X$ has at least two conjugacy classes of non-identity elements, then there exists $a\in A\setminus\{1\}$ and $b\in B\setminus \{1\}$ such that $ab$ is a derangement. If $X$ has only one conjugacy class of non-identity elements then $|X|=2$.
\end{proof}

\begin{theorem}[{Saxl, see~\cite[Proposition~2]{Saxl}}]
\label{thrm:saxl}Let $T$ be a non-abelian simple group and let $M$ be a subgroup of $T$. If
$$T=\bigcup_{\varphi\in\mathrm{Aut}(T)}M^\varphi,$$
then $T=M$.
\end{theorem}

\end{notation}

We have the following observation.

\begin{lemma}\label{l:prel1}
Let $G$ be a transitive permutation group on $\Omega$ and let $\Sigma$ be a system of imprimitivity for the action of $G$ on $\Omega$. If $C\subseteq G$ is a clique for the derangement graph $\Gamma_{G,\Sigma}$ of $G$ in its action on $\Sigma$, then $C$ is also a clique for the derangement graph $\Gamma_{G,\Omega}$ of $G$ in its action on $\Omega$.  
\end{lemma}
\begin{proof}
Let $C\subseteq G$ be clique for the derangement graph of $G$ in its action on $\Sigma$. 
Now, given two distinct elements $g$ and $h$ of $C$, we have that $gh^{-1}$ fixes no element of $\Sigma$. Since $\Sigma$ is a system of imprimitivity for $\Omega$, $gh^{-1}$ cannot fix any element of $\Omega$. Therefore $C$ is a clique for the derangement graph of $G$ in its action on $\Omega$.  
\end{proof}
The modern key for analyzing a finite primitive permutation group $G$ is to study
the \textit{\textbf{socle}} $N$ of $G$, that is, the subgroup generated by the minimal normal subgroups of $G$. The socle of a non-trivial finite group is isomorphic to the non-trivial
direct product of simple groups; moreover, for finite primitive groups, these simple groups are pairwise isomorphic. The O'Nan-Scott theorem describes in detail
the embedding of $N$ in $G$ and collects some useful information about the action
of $N$. In~\cite[Theorem]{LPSLPS}, five types of primitive groups are defined (depending
on the group- and action-structure of the socle), namely HA (Affine), AS (Almost
Simple), SD (Simple Diagonal), PA (Product Action) and TW (Twisted Wreath),
and it is shown that every primitive group belongs to exactly one of these types.
We remark that in~\cite{19} this subdivision into types is refined, namely the PA type
in~\cite{LPSLPS} is partitioned in three parts, which are called  HC
(Holomorphic Compound), CD (Compound Diagonal) and PA. Moreover, the SD type is partitioned in two parts, which are called HS (Holomorphic Simple) and SD (Simple Diagonal). For what follows,
we find it convenient to use this subdivision into eight types of the finite primitive
permutation groups.\footnote{This division has the advantage that there are no overlaps between the eight O'Nan-Scott types of primitive permutation groups.}

\begin{lemma}\label{l:prel2}
Let $G$ be a transitive permutation group having domain $\Omega$, let $\Sigma$ be a system of imprimitivity and let $\Delta\in \Sigma$. If the action of $G_{\{\Delta\}}$ on $\Delta$ is primitive, then so is the action of $G_{\{{\Delta'}\}}$ on ${\Delta'}$, $\forall {\Delta'}\in\Sigma$. Furthermore, if $G_{\{\Delta\}}$ has O'Nan-Scott type $X$, 
then so does the action of $G_{\{{\Delta'}\}}$ on ${\Delta'}$, $\forall {\Delta'}\in\Sigma$.
\end{lemma}
\begin{proof}
This follows from the fact that, for every ${\Delta'}\in\Sigma$, there exists $g\in G$ with $\Delta^g={\Delta'}$. In fact, this implies $G_{\{{\Delta'}\}}=G_{\{\Delta^g\}}=(G_{\{\Delta\}})^g$, from which the lemma immediately follows.
\end{proof}

For improving the flow of the argument in the proof of Proposition~\ref{prop:technique1}, we prove the following.

\begin{lemma}\label{prop:technique1aux}
Let $G$ be a transitive permutation group on $\Omega$, let $\Sigma$ be a system of imprimitivity for the action of $G$ on $\Omega$ and let $\Delta\in \Sigma$. Suppose 
\begin{itemize}
\item $G_{\{\Delta\}}$ acts primitively on $\Delta$, 
\item $G_{(\Sigma)}\ne 1$ and 
\item $G_{(\Delta)}=1$.
\end{itemize}
Let $M$ be a minimal normal subgroup of $G$ contained in $G_{(\Sigma)}$. 
 Then 
 \begin{enumerate}
 \item\label{technique1aux1} either $M$ is a minimal normal subgroup of $G_{\{\Delta\}}$ for every $\Delta\in \Sigma$, or 
 \item\label{technique1aux2} $M$ is the socle of $G_{\{\Delta\}}$ and is the direct product of two isomorphic subgroups $M_{1,\Delta},M_{2,\Delta}$ acting regularly on $\Delta$, for every $\Delta\in \Sigma$.\footnote{Observe that the two isomorphic subgroups of $M$ might depend on $\Delta$.}
 \end{enumerate}
\end{lemma}
\begin{proof}
As $G_{(\Delta)}=1$, the group $G_{\{\Delta\}}$ acts faithfully on $\Delta$. Now, fix $\Delta\in \Sigma$.

Assume first that $M$ is abelian. Since $G_{\{\Delta\}}$ is primitive on $\Delta$ and $M\unlhd G_{\{\Delta\}}$, we deduce that $M$ is transitive on $\Delta$ and, since $G_{\{\Delta\}}$ is faithful on $\Delta$, we deduce that $M$ is regular on $\Delta$. Since this argument does not depend on $\Delta$, we get that $M$ is  the socle of $G_{\{\Delta'\}}$ and is the unique minimal normal subgroup of $G_{\{\Delta'\}}$, for every $\Delta'\in \Sigma$. In particular, part~\eqref{technique1aux1} holds.

Assume now that $M$ is non-abelian and write $M=T_1\times \cdots \times T_\ell$, for some positive integer $\ell$ and for some pairwise isomorphic non-abelian simple groups $T_1,\ldots,T_\ell$. Let $M_1$ be a minimal normal subgroup of $G_{\{\Delta\}}$ with $M_1\le M$. Therefore, relabeling the indexed set $\{1,\ldots,\ell\}$ if necessary, we may suppose that $M_1=T_1\times \cdots \times T_\kappa$, for some $1\le\kappa\le\ell$. If $\kappa=\ell$, then $M$ is a minimal normal subgroup of $G_{\{\Delta\}}$. However, as $M\unlhd G$ and $G$ is transitive on $\Sigma$, we deduce that $M$ is a minimal normal subgroup of $G_{\{\Delta'\}}$, for every $\Delta'\in \Sigma$. Therefore, part~\eqref{technique1aux1} holds. Suppose $\kappa<\ell$. Let $M_2$ be a minimal normal subgroup of $G_{\{\Delta\}}$ with $T_{\kappa+1}\le M_2\le M$. Using the structure of minimal normal subgroups in primitive groups, we deduce that $M_1\cong M_2$ and hence $\ell=2\kappa$. Therefore $M_2=T_{\kappa+1}\times \cdots \times T_{2\kappa}$. In this case, $M_1$ and $M_2$ both act regularly on $\Delta$ and $M=M_1\times M_2$. In particular, $M$ is the socle of $G_{\{\Delta\}}$. As above, since $M\unlhd G$ and $G$ is transitive on $\Sigma$, $M$ is the socle of $G_{\{\Delta'\}}$, for every $\Delta'\in \Sigma$. Thus part~\eqref{technique1aux2} holds.
\end{proof}

\begin{proposition}\label{prop:technique1}
Let $G$ be a transitive permutation group on $\Omega$, let $\Sigma$ be a system of imprimitivity for the action of $G$ on $\Omega$ and let $\Delta\in \Sigma$. Suppose 
\begin{itemize}
\item $G_{\{\Delta\}}$ acts primitively on $\Delta$, 
\item $G_{(\Sigma)}\ne 1$ and 
\item $G_{(\Delta)}=1$.
\end{itemize} Then there exists $g\in G_{(\Sigma)}$ acting as a derangement on $\Omega$.
\end{proposition}
\begin{proof}
The hypothesis $G_{(\Delta)}= 1$ implies that the action of $G_{\{\Delta\}}$ on $\Delta$ is faithful. Whereas, the condition $G_{(\Sigma)}\ne 1$ implies that $G$ does not act faithfully on the system of imprimitivity $\Sigma$, that is, there exists $1\ne g\in G$ with ${\Delta'}^g={\Delta'}$, $\forall {\Delta'}\in\Sigma$.

Since $G_{(\Delta)}=1$ and since $G_{(\Sigma)}\le G_{\{\Delta\}}$, the group $G_{(\Sigma)}$ acts faithfully on $\Delta$.

We use the O'Nan-Scott theorem. Let  ${\Delta'}\in \Sigma$ and let $g\in G$ with ${\Delta'}=\Delta^g$. According to Lemma~\ref{l:prel2}, the O'Nan-Scott type of the group $G_{\{\Delta\}}$ acting on $\Delta$ is identical to that of $G_{\{{\Delta'}\}}$ acting on ${\Delta'}$.

Let $S$ be a minimal normal subgroup of $G$ contained in $G_{(\Sigma)}$. We use Lemma~\ref{prop:technique1aux}.% contained in $N$. As $G_{\{\Delta\}}$ is primitive on $\Delta$, $G_{\{\Delta\}}$ has either one or two minimal normal subgroups. Furthermore, in the case that $G_{\{\Delta\}}$ has two minimal normal subgroups, these have the same order. Observe that $S\le N\le G_{\{{\Delta'}\}}$  and $S$ is a minimal normal subgroup of $G_{\{{\Delta'}\}}$\footnote{If $S$ is not a minimal normal subgroup of $G_{\{{\Delta'}\}}$, then there exists a minimal normal subgroup $S'$ of $G_{\{{\Delta'}\}}$ with $S'<S$. Given $g\in G$ with ${\Delta'}^g=\Delta$, we get that $S'^g$ is a minimal normal subgroup of $G_{\{{\Delta'}\}}^g=G_{\{\Delta\}}$ having order different from $|S|$, which is a contradiction.}, $\forall {\Delta'}\in \Sigma$.

When $G_{\{\Delta\}}$ (in its action on $\Delta$) has O'Nan-Scott type HA or TW, $S$ is the socle of $G_{\{{\Delta'}\}}$ and hence it acts regularly on ${\Delta'}$, for every ${\Delta'}\in \Sigma$. Therefore, any non-identity element of $S$ is a derangement on $\Omega$.

When $G_{\{\Delta\}}$  (in its action on $\Delta$)
has O'Nan-Scott type HS or HC, $S=S_{1,\Delta'}\times S_{2,\Delta'}$, where $S_{1,\Delta'}$ and $S_{2,\Delta'}$ are the two minimal normal subgroups of $G_{\{{\Delta'}\}}$ generating the socle of $G_{\{{\Delta'}\}}$, $\forall {\Delta'}\in \Sigma$. Moreover, from the structure of primitive groups of HS and HC type, every simple direct factor of $S$  acts semiregularly on ${\Delta'}$, for each ${\Delta'}\in \Sigma$. Therefore, any non-identity element of $S$ lying on a simple direct factor of $S$ is a derangement on $\Omega$.

Suppose $G_{\{\Delta\}}$  (in its action on $\Delta$) has O'Nan-Scott type SD or CD.  Then $$S=T_1\times\cdots \times T_{\ell},$$ where $T_1,\ldots,T_{\ell}$ are pairwise isomorphic non-abelian simple groups.  Observe that, from the structure of primitive groups of  O'Nan-Scott type SD and CD, for each $i\in \{1,\ldots,\ell\}$, the group $T_i$ acts semiregularly on the domain ${\Delta'}$, for each ${\Delta'}\in\Sigma$. In particular, from this it follows that, for every $i$, the non-identity elements of $T_i$ are derangements on $\Omega$.

Suppose that $G_{\{\Delta\}}$ is of type AS or PA. In particular, $S$ is the unique minimal normal subgroup of $G_{\{\Delta'\}}$ and is the socle of $G_{\{{\Delta'}\}}$, $\forall{\Delta'}\in\Sigma$. Suppose that $S$ has no derangement in its action on $\Omega$. Let $\omega_0\in \Delta$. Then
\begin{equation}\label{eq:A}S=\bigcup_{\omega\in \Omega}S_\omega=\bigcup_{g\in G}S_{\omega_0}^g.\end{equation}
For every $g\in G$, we have $S^g=S$ and hence the action of $g$ by conjugation on $S$ induces an automorphism of $S$. Therefore, from~\eqref{eq:A}, we have
\begin{equation}\label{eq:B}S=\bigcup_{\varphi\in \mathrm{Aut}(S)}S_{\omega_0}^\varphi.
\end{equation}
From the definition of primitive groups of AS and PA type, the $G_{\{\Delta\}}$-set $\Delta$ admits a Cartesian decomposition and hence $\Delta=\Lambda^\ell$, for some finite set ${\Lambda}$ with $|{\Lambda}|\ge 5$ and for some $\ell\ge 1$ (the value $\ell=1$ corresponds to the case that $G_{\{\Delta\}}$ is of AS type). Moreover,
$$G_{\{\Delta\}}\le H\mathrm{wr}\mathrm{Sym}(\ell),$$
where $H\le \mathrm{Sym}({\Lambda})$, $H$ is almost simple and acts primitively on ${\Lambda}$, and $G_{\{\Delta\}}$ is endowed of the product action. We identify the elements of $\Delta$ with the elements of ${\Lambda}^\ell$ and we identify $G_{\{\Delta\}}$ with its image in the embedding in $H\mathrm{wr}\mathrm{Sym}(\ell)$. Here, the socle $S$ of $G_{\{\Delta\}}$ is $T^\ell$, where $T$ is the socle of $H$. Replacing the element $\omega_0\in \Delta$ if necessary, we may suppose that $\omega_0=(\lambda,\ldots,\lambda)$, for some $\lambda\in \Lambda$. Therefore
$$S_{\omega_0}=T_{\lambda}\times \cdots \times T_{\lambda}=T_{\lambda}^\ell.$$
Since $\mathrm{Aut}(S)=\mathrm{Aut}(T^\ell)=\mathrm{Aut}(T)\mathrm{wr}\mathrm{Sym}(\ell)$, from~\eqref{eq:B}, we deduce
\begin{equation}\label{eq:C}
T^\ell=S=\bigcup_{\varphi\in\mathrm{Aut}(S)}S_{\omega_0}^\varphi=\bigcup_{\varphi\in \mathrm{Aut}(T^\ell)}(T_{\lambda}^\ell)^\varphi=
\bigcup_{\varphi\in \mathrm{Aut}(T)\mathrm{wr}\mathrm{Sym}(\ell)}(T_{\lambda}^\ell)^\varphi.
\end{equation}
Since $T_{\lambda}^\ell$ is normalized by $\mathrm{Sym}(\ell)$, that is, $$(T_{\lambda}^\ell)^\varphi=T_{\lambda}^\ell,\,\hbox{   }\forall \varphi\in\mathrm{Sym}(\ell),$$ 
from~\eqref{eq:C}, we deduce
\begin{equation}\label{eq:D}T^\ell=\bigcup_{\varphi_1,\ldots,\varphi_\ell\in \mathrm{Aut}(T)}T_{\lambda}^{\varphi_1}\times \cdots \times T_{\lambda}^{\varphi_\ell}.\end{equation}
In particular, by considering only the first direct factor of $T^\ell$, from~\eqref{eq:D}, we get
$$T=\bigcup_{\varphi\in\mathrm{Aut}(T)}T_{\lambda}^\varphi.$$
By a celebrated result of Saxl~\cite{Saxl}, this is not possible: see Theorem~\ref{thrm:saxl}.
\end{proof}

\subsection{Hypergraphs}\label{hypergraphs}
We need some basic definitions on hypergraphs, because this terminology is needed in the statement and in the proof of Proposition~\ref{prop:technique2}. A \textit{\textbf{hypergraph}} is an ordered pair $\Gamma=(V,E)$, where $V$ is a finite set and $E$ is a collection of subsets of $V$. We are only interested in $k$-\textit{\textbf{uniform}} hypergraphs, that is, hypergraphs where each element in $E$ has cardinality $k$, for a fixed non-negative integer $k$. 

Up to this point, everything is standard, but in our application, we need a further twist. Let $a$ and $b$ be positive integers with $a\ge 2$. An $(a,b)$-\textit{\textbf{hypergraph}} is an ordered pair $\Gamma=(V,E)$, where $V$ is a finite set and where the elements $e$ of $E$ are partitions of $ab$-subsets of $V$ into $b$ parts each having cardinality $a$. For instance, when $a=2$ and $b=3$, a typical element of $E$ is of the form
$$\{\{v_1,v_2\},\{v_3,v_4\},\{v_5,v_6\}\}.$$
In the special case that $b=1$, our definition returns the notion of $a$-uniform hypergraph.

An \textit{\textbf{automorphism}} $\varphi$ of the $(a,b)$-hypergraph $\Gamma$ is a permutation of $V$ preserving $E$, that is, $e^\varphi = e$ for every $e\in E$. Observe that $\varphi$ fixes $e$ as a partition and hence $\varphi$ is allowed to permute the parts within $e$. Let $G$ be a group acting on $\Gamma$ by automorphisms (in particular, we are not insisting that $G$ acts faithfully on $V$). We say that $\Gamma$ is $G$-\textit{\textbf{vertex transitive}} if $G$ acts transitively on $V$ and, analogously, we say that $\Gamma$ is $G$-\textit{\textbf{edge transitive}} if $G$ acts transitively on $E$. 

In Proposition~\ref{prop:technique2} we are interested only in certain  $(a,b)$-hypergraphs. We say that $G$ is \textit{\textbf{special}} for $\Gamma$ if
\begin{itemize}
\item $\Gamma$ is $G$-vertex and $G$-edge-transitive, and
\item for every edge $e=\{p_1,\ldots,p_b\}\in E$, the edge stabilizer $G_e$ is transitive on the $b$ parts of $e$, and either
\begin{itemize}
\item $G_e$ is also transitive on $p_1\cup\cdots\cup p_b$, that is, $G_e$ is transitive on the underlying set of vertices appearing in $e$, or
\item $a=2$. 
\end{itemize} 
\end{itemize}

Given a positive integer $c$, a \textit{\textbf{colouring}} of the $(a,b)$-hypergraph $\Gamma$ with $c$ colours is a function $\eta:V\to \{1,\ldots,c\}$ such that there exists no $e=\{p_1,\ldots,p_b\}\in E$ where $p_i$ is \textit{\textbf{monochromatic}}\footnote{A subset of $V$ is monochromatic if all of its elements have the same image via $\eta$.} for every $i$. In other words, for every $e=\{p_1,\ldots,p_b\}$, there exists $i\in \{1,\ldots,b\}$ such that $p_i$ is not monochromatic. The \textit{\textbf{chromatic number}} $\chi(\Gamma)$ of $\Gamma$ is the minimum $c$ such that $\Gamma$ admits a colouring with $c$ colours. In the special case that $b=1$ and $a=2$, our definition of chromatic number coincides with the classic definition of chromatic number of a graph.

Proposition~\ref{prop:technique2} complements Proposition~\ref{prop:technique1}. Indeed, while Proposition~\ref{prop:technique1} assumes $G_{(\Delta)}= 1$, Proposition~\ref{prop:technique2} extends this by considering the case  $G_{(\Delta)}\ne 1$, albeit under a marginally stronger assumption   on the permutation group induced by $G_{\{\Delta\}}$ on $\Delta$.%\footnote{The proof of Proposition~\ref{prop:technique2} can be traced back to the proof of~\cite[Proposition~3.1]{Praeger} and of~\cite[Proposition~2.1]{Praegerr}. We include our own proof here for two reasons. First, the notation is quite different in our context and is tailored to the application to permutation groups, which would make cumbersome to use directly~\cite{Praeger,Praegerr}. Second, our hypotheses in Proposition~\ref{prop:technique2} are slightly different, which make some of the cases arising in~\cite[Proposition~3.1]{Praeger} not applicable in our context. We take no credit in the novelty  of Proposition~\ref{prop:technique2} and indeed, the brilliant idea of using an auxiliarly graph is taken from~\cite[Corollary~3.2]{Praeger}.}
\begin{proposition}\label{prop:technique2}
Let $G$ be a permutation group on $\Omega$, let $\Sigma$ be a system of imprimitivity for the action of $G$ on $\Omega$ and let $\Delta\in \Sigma$. Suppose
\begin{itemize}
\item  $G_{\{\Delta\}}$ acts primitively on $\Delta$,
\item $G_{(\Sigma)}\ne 1$,
\item $G_{(\Delta)}\ne 1$, and
\item the primitive permutation group induced by $G_{\{\Delta\}}$ on $\Delta$ has non-abelian socle. 
\end{itemize}
Let $N$ be the socle of $G_{(\Sigma)}$ and  let $\omega\in \Delta$.
Then either there exists $g\in G_{(\Sigma)}$ acting as a derangement on $\Omega$, or $N=S_1\times \cdots\times S_\ell$ such that
\begin{enumerate}
\item\label{technique2:nr1} $S_i$ is a minimal normal subgroup of $G_{(\Sigma)}$ for each $i\in \{1,\ldots,\ell\}$, and $\ell\ge 5$,
\item\label{technique2:nr2} $S_i\cong T^\kappa$, for some non-abelian simple group $T$ and some $\kappa\ge 1$, and if $\ell=5$, then $T\cong \mathrm{Alt}(5)$,
\item\label{technique2:nr3} $G$ acts transitively by conjugation on $\{S_1,\ldots,S_\ell\}$,
\item\label{technique2:nr4} there exists  an $(a,b)$-hypergraph $\Gamma$ such that
\begin{enumerate}
\item $\Gamma$ has vertex set $\{S_1,\ldots,S_\ell\}$,
\item\label{technique2:nr42}  $\chi(\Gamma)\ge t+1$, where $t$ is the number of $\mathrm{Aut}(T^\kappa)$-conjugacy classes in $T^\kappa$
\item  $G/G_{(\Sigma)}$ is special  for $\Gamma$,
\item $G_{\{\Delta\}}$ fixes some  edge $e=\{p_1,\ldots,p_b\}$ of $\Gamma$,
\end{enumerate}
\item\label{technique2:nr5} we have $$N\cap G_\omega=\prod_{i=1}^b\mathrm{Diag}\left(\prod_{x\in p_i}S_x\right)\times \prod_{\substack{k=1\\k\notin p_1\cup\cdots\cup p_b}}^\ell S_k.$$
\end{enumerate}
\end{proposition}
\begin{proof}
Let $W$ be a maximal subgroup of $G_{(\Sigma)}$ with $G_\omega\cap G_{(\Sigma)}\le W$.
The proof follows combining Proposition~3.1 and Corollary~3.2 in~\cite{Praeger}. We give some details to see how to deduce our proposition from the work in~\cite[Section~3]{Praeger}: we use the notation therein in what follows. Set $H:=G_{(\Sigma)}$ and $A:=G$ and assume that $H=G_{(\Sigma)}$ has no derangement. This means that
$$H=\bigcup_{a\in A}W^a,$$
because $H_\omega=G_{(\Sigma)}\cap G_\omega\le W$ by hypothesis. Therefore, using the terminology in~\cite{Praeger}, $W$ is an $A$-covering subgroup of $H$. 
Let $K=\bigcap_{a\in A}W^a$. Assume $K\ne 1$. Since $G_{\{\Delta\}}$ is primitive on $\Delta$ and since $K\unlhd G_{\{\Delta\}}$, we deduce that either $K$ is transitive on $\Delta$, or $K\le G_{(\Delta)}$. The latter possibility is excluded by the fact that $1\ne K\unlhd G$. Therefore,  $K$ is transitive on $\Delta$ and hence $G_{\{\Delta\}}=KG_\omega$. Intersecting both sides by $G_{(\Sigma)}$ and using the modular law, we get $G_{(\Sigma)}=K(G_{(\Sigma)}\cap G_\omega)\le KW\le W$, which contradicts the fact that $W\ne G_{(\Sigma)}$. This shows that $\bigcap_{a\in A}W^a=K=1$ and hence the hypothesis of Proposition~3.1 in~\cite{Praeger} are all satisfied. Therefore parts~\eqref{technique2:nr1}--~\eqref{technique2:nr3} follow immediately from~\cite[Proposition~3.1 and Corollary~3.2]{Praeger}. Proposition~3.1 in~\cite{Praeger} also gives an explicit description of $N\cap W$, which we now use to prove the remaining parts in the statement.

It follows from~\cite[Theorem~4.6A]{dixon} and~\cite[Proposition~3.1 part~(iii)]{Praeger} that $G_{\{\Delta\}}/G_{(\Delta)}$ is embedded in 
$$H\,\mathrm{wr}\,\mathrm{Sym}(b),$$
with the primitive product action, where $H$ is a primitive group of SD type or HS type. We are allowing here $b=1$. 

We consider the function $\theta $ having domain $\Sigma$ and codomain the set of all subsets of $\{S_1,\ldots,S_\ell\}$. For each $\Delta'\in \Sigma$, we let $\theta(\Delta')=\{S_i\mid S_i\nleq G_{(\Delta')}\}$. Since $G$ acts transitively on $\Sigma$, the sets $\theta(\Delta')$ have all the same cardinality, say $k$. Moreover, from the O'Nan-Scott type of $G_{\{\Delta\}}/G_{(\Delta)}$ that we have highlighted in the previous paragraph, we deduce that $\theta(\Delta')$ is endowed of a natural partition into $b$ sets of cardinality $k/b$. Let $V=\{S_1,\ldots,S_\ell\}$, let $a=k/b$ and let $E$ be the image of $\theta$. In particular, $\Gamma=(V,E)$ is an $(a,b)$-hypergraph. Since  the action of $G$ by conjugation on $V$ is transitive and since $G$ is transitive on $\Sigma$, we deduce that $\Gamma$ is $G$-vertex and $G$-edge transitive. Moreover, $G_{\{\Delta\}}$ fixes the edge $\theta(\Delta)$ and, from the O'Nan-Scott type of $G_{\{\Delta\}}/G_{(\Delta)}$ we deduce that  $G/G_{(\Sigma)}$ is special for $\Gamma$. Before discussing~\eqref{technique2:nr42}, we observe that~\eqref{technique2:nr5} follows from the way that $\Gamma$ is defined and again from the O'Nan-Scott type of $G_{\{\Delta\}}/G_{(\Delta)}$.

Let $t$ be the number of $\mathrm{Aut}(T^\kappa)$-conjugacy classes in $T^\kappa$ and let $x_1,\ldots,x_t$ be a set of representatives.
Let $c$ be the chromatic number of $\Gamma$. Suppose $c\le t$. In particular, there exists a colouring $\eta$ of $V$ where we can use as ``colours'' the elements in $\{x_1,\ldots,x_t\}$. Using the colouring $\eta$, consider $$g=({\eta(S_1)},{\eta(S_2)},\ldots,{\eta(S_\ell)})\in N,$$
in other words, we are considering the element of $N$ where in the $i^{\mathrm{th}}$th coordinate we put the element having colour $\eta(S_i)$. Since we are dealing with the case that $G_{(\Sigma)}$ has no derangements, we deduce that $g$ fixes some point of $\Omega$, because $N\le G_{(\Sigma)}$. Without loss of generality $g$ fixes $\omega$ and hence $g\in N\cap G_\omega$. However, from~\eqref{technique2:nr5}, we see that the coordinates of $g$ belonging to a same part of the edge $\theta(\Delta)$ must be conjugate under some automorphism of $T^\kappa$, that is, all the parts of the edge $\theta(\Delta)$ are monochromatic, contradicting the fact that $\eta$ is a colouring of $\Gamma$. This contradiction has arisen from assuming $c\le t$ and hence~\eqref{technique2:nr42} follows.
 \end{proof}

\begin{remark}{\rm 
The hypotheses in Proposition~\ref{prop:technique2} seem rather unnatural, but they are actually necessary. Indeed, consider for instance the non-abelian simple group $T=\mathrm{Alt}(5)$, let $G=\mathrm{Aut}(T)\mathrm{wr}\mathrm{Sym}(5)$ and let $N=T^{4}$ be the base group of $G$. Let $H=\{(t_1,t_2,t_3,t_4,t_5)\in N\mid t_1=t_2\}\cong T^{5}$ and let $\Omega$ be the set of right cosets of $H$ in $G$. Thus $|\Omega|=|\mathrm{Sym}(5)|[N:H]=120\cdot 60^2$ and $G$ can be viewed as a permutation group on $\Omega$. 

Let $\Sigma$ be the system of imprimitivity given by the orbits of $N$ on $\Omega$. Thus, $G_{(\Sigma)}=N\ne 1$. For every $\Delta\in \Sigma$, we have $G_{\{\Delta\}}=N$ because $G$ acts regularly on $\Sigma$. Moreover, as $H$ is a maximal subgroup of $N$, we deduce that $G_{\{\Delta\}}$ acts primitively on $\Delta$. Finally, as the core of $H$ in $N$ is isomorphic to $T^{3}$, we deduce $G_{(\Delta)}\cong T^{3}\neq 1$.

Let $(t_1,t_2,t_3,t_4,t_5)\in N=G_{(\Sigma)}$. As $T$ has only four $\mathrm{Aut}(T)$-conjugacy classes, there exist two distinct indices $i,j\in\{1,\ldots,5\}$ with $t_i^{\mathrm{Aut}(T)}=t_j^{\mathrm{Aut}(T)}$. Since $\mathrm{Sym}(5)$ acts $2$-transitively on the $5$ simple direct factors of $N$, there exists $\sigma\in G$ with $(t_1,t_2,t_3,t_4,t_5)^\sigma\in H$. This shows that every element of $G_{(\Sigma)}$ fixes some point of $\Omega$ and hence no element of $G_{(\Sigma)}$ acts as a derangement on $\Omega$.

In this example, the auxiliary $(a,b)$-hypergraph arising in Proposition~\ref{prop:technique2} part~\eqref{technique2:nr4} has parameters $a=2$ and $b=1$, and is the complete graph on $5$ vertices.
}
\end{remark}

We conclude this section with a combinatorial lemma that is needed when we apply Proposition~\ref{prop:technique2}.
\begin{lemma}\label{l:conjclasses}
Let $T$ be a non-abelian simple group and let $t$ be the number of $\mathrm{Aut}(T)$-conjugacy classes of elements of $T$, and let $\kappa$ be a positive integer.  Then ${t+\kappa-1\choose \kappa}$ is the number of $\mathrm{Aut}(T^\kappa)$-conjugacy classes of elements of $T^\kappa$.
\end{lemma}
\begin{proof}
This follows from classical problem in combinatorics, see~\cite[Section~3.4 and Theorem~3.4.2]{lovaz}. Let $x_1,\ldots,x_t$ be representatives for the $\mathrm{Aut}(T)$-classes of elements of $T$. Every element in $T^\kappa$ is $\mathrm{Aut}(T^\kappa)$-conjugate to an element $(y_1,\ldots,y_\kappa)$, where $y_1,\ldots,y_\kappa\in \{x_1,\ldots,x_t\}$. Moreover, as $\mathrm{Aut}(T^\kappa)=\mathrm{Aut}(T)\mathrm{wr}\mathrm{Sym}(\kappa)$, the order of the elements appearing in the coordinates of  $(y_1,\ldots,y_\kappa)$ does not change its $\mathrm{Aut}(T^\kappa)$-conjugacy class. Therefore, using the terminology in~\cite{lovaz}, the number of $\mathrm{Aut}(T^\kappa)$-conjugacy classes in $T^\kappa$ equals the number of ways to distribute $\kappa$ identical pennies to $t$ children. 
\end{proof}

\section{Step 1}\label{sec:step1}
We start by restating Theorem~\ref{thrm:main}.

\smallskip

\noindent\textbf{Theorem }
Let $G$ be a finite transitive permutation group on $\Omega$. Then either the derangement graph of $G$ in its action on $\Omega$ has a clique of size at least $4$, or one of the following holds
\begin{enumerate}
\item $|\Omega|\le 3$ and $\mathrm{Alt}(\Omega)\le G\le\mathrm{Sym}(\Omega)$,
\item$|\Omega|=6$ and $G\cong\mathrm{Alt}(4)$,
\item$|\Omega|=18$, $|G|=324$ and $G$ is the group $(18,142)$ in the library of transitive groups in \texttt{magma},
\item$|\Omega|=30$, and $G$ is the group $(30,126)$ or $(30,233)$ in the library of transitive groups in \texttt{magma}.
\end{enumerate}

\smallskip

Let $G$ be a transitive permutation group having domain $\Omega$ and suppose that the derangement graph $\Gamma_{G,\Omega}$ of $G$ in its action on $\Omega$ has no cliques of size $4$. If $\Gamma_{G,\Omega}$ has no triangles, then the result follows from~\cite{KRS}, indeed, part~\eqref{thrm:maineq1} is satisfied with $|\Omega|\le 2$. Therefore, we may suppose that $\Gamma_{G,\Omega}$ has at least one triangle. In particular, $|\Omega|>2$. Our aim is to prove that $|\Omega|\le 30$. Thus, the satisfaction of part~\eqref{thrm:maineq1}, or~\eqref{thrm:maineq2}, or~\eqref{thrm:maineq3}, or~\eqref{thrm:maineq4} will immediately follow.

In the proof of Theorem~\ref{thrm:main}, we argue by induction on $|G|$. Specifically, by substituting $G$ with a proper transitive subgroup  if necessary, we may assume that 
\begin{equation}\label{assumption1}
G \textrm{ is minimally transitive},
\end{equation}
meaning that every proper subgroup of $G$ acts intransitively on $\Omega$.\footnote{Of course, after determining these minimally transitive groups, we need to establish which over groups lack cliques of size $3$, to complete the proof of Theorem~\ref{thrm:main}.}

The objective of this section  is to establish this preliminary result towards proving Theorem~\ref{thrm:main}.
\begin{lemma}\label{l:step1}
The group $G$ possesses a system of imprimitivity $\Sigma_1$ with $|\Sigma_1|=3$. The permutation group induced by $G$ on $\Sigma_1$ is $\mathrm{Alt}(\Sigma_1)$ and $G_{(\Sigma_1)}$ has no derangements. See Figure~$\ref{figure1}$ for some information on the subgroup lattice of $G$.
\end{lemma}
\begin{proof}
Let $\Sigma_1$ be a system of imprimitivity for the action of $G$ on $\Omega$ with $|\Sigma_1|>1$ and having minimal cardinality. In particular, $G$ acts primitively on $\Sigma_1$. 
Consider the natural homomorphism
$$\bar{\,}:G\to \mathrm{Sym}(\Sigma_1)$$ arising from the action of $G$ on $\Sigma_1$. Specifically, $G_{(\Sigma_1)}$ is the kernel of this mapping.

Since the derangement graph of $G$ on $\Omega$ lacks cliques of size $4$, by Lemma~\ref{l:prel1}, $\Gamma_{\bar G,\Sigma_1}$ also lacks such cliques.

Now, suppose the derangement graph $\Gamma_{\bar{G},\Sigma_1}$ has no triangles. As $|\Sigma_1|>1$, according to~\cite{KRS}, we have $|\bar{G}|=|\Sigma_1|=2$. Consequently, $[G:G_{(\Sigma_1)}]=2$ and $\bar{G}\cong\mathrm{Sym}(2)$. Notice that each element in $G\setminus G_{(\Sigma_1)}$ acts as a derangement on $\Sigma_1$\footnote{This follows from the fact that, if $g\in G\setminus G_{(\Sigma_1)}$, then $g$ swaps the two blocks in the system of imprimitivity $\Sigma_1$ and hence $g$ fixes no point of $\Omega$.} and thus on $\Omega$ as well.

Consider $g\in G\setminus G_{(\Sigma_1)}$. If $G_{(\Sigma_1)}$ contains a derangement $k$ in its action on $\Omega$, then $gk$ and $gkg^{-1}$ are both derangements since $gk\in G\setminus G_{(\Sigma_1)}$ and $gkg^{-1}$ is conjugate to $k$. Consequently, 
$$C=\{1,g,k,gk\},$$
 forms a clique of size four in $\Gamma_{G,\Omega}$, contradicting the assumption that $\Gamma_{G,\Omega}$ has no $4$-cliques. Hence, $G_{(\Sigma_1)}$, in its action on $\Omega$, contains no derangements. Consequently, the derangement graph $\Gamma_{G,\Omega}$ is bipartite with bipartition $$G_{(\Sigma_1)},G\setminus G_{(\Sigma_1)}.$$
  However, this contradicts the fact that $\Gamma_{G,\Omega}$ admits a triangle.

This contradiction demonstrates that the derangement graph $\Gamma_{\bar{G},\Sigma_1}$ of $\bar G$ has a triangle but no $4$-clique. Since $\bar{G}$ is primitive, by~\cite[Theorem~1.2]{FPS}, we conclude that $|\Sigma_1|=3$ and $\mathrm{Alt}(\Sigma_1)\le \bar G\le \mathrm{Sym}(\Sigma_1)$.

Now, let $G'$ be the preimage under $\bar{\,}$ of $\mathrm{Alt}(\Sigma_1)$. We have $[G:G']=1$ when $\bar G=\mathrm{Alt}(\Sigma_1)$ and $[G:G']=2$ when $\bar G=\mathrm{Sym}(\Sigma_1)$. Since $G'$ acts transitively on $\Sigma_1$, it follows that $G'$ is also transitive on $\Omega$\footnote{This implication is clear when $G=G'$ and it requires a moment's thought when $|G:G'|=2.$ Indeed, suppose $|G:G'|=2$. If $G_{\omega_1}\le G'$, then from the bijection between systems of imprimitivity of $G$ on $\Omega$ and the interval lattice of $G/G_{\omega_1}$ we deduce that $G$ has a system of imprimitivity of size $2$, because $|G:G'|=2$ and $G_{\omega_1}\le G'$. However, this contradicts the fact that we have shown that the minimal cardinality of a system of imprimitivity for $G$ is $3$. 
 Therefore, $G_{\omega_1}\nleq G'$. Since $|G:G'|=2$, we have $G=G_{\omega_1}G'$.  We deduce by the Frattini argument that $G'$ is transitive on $\Omega$.} Therefore, by~\eqref{assumption1}, we conclude that $G=G'$, implying
\begin{equation*}
\bar G=\mathrm{Alt}(\Sigma_1).
\end{equation*}

Suppose $G_{(\Sigma_1)}$ contains a derangement $k$ in its action on $\Omega$. Now, let $c\in G$ such that $\bar{c}$ acts as a cycle of length $3$ on $\Sigma_1$, i.e., $\bar G=\mathrm{Alt}(\Sigma_1)=\langle\bar{c}\rangle$. Following the previous argument, we observe that 
$$C=\{1,c,c^2,k,ck,c^2k\}$$ forms a clique of size six in $\Gamma_{G,\Omega}$, contradicting the assumption that $\Gamma_{G,\Omega}$ has no $4$-cliques. Therefore, $G_{(\Sigma_1)}$ in its action on $\Omega$ contains no derangement.
\end{proof} 

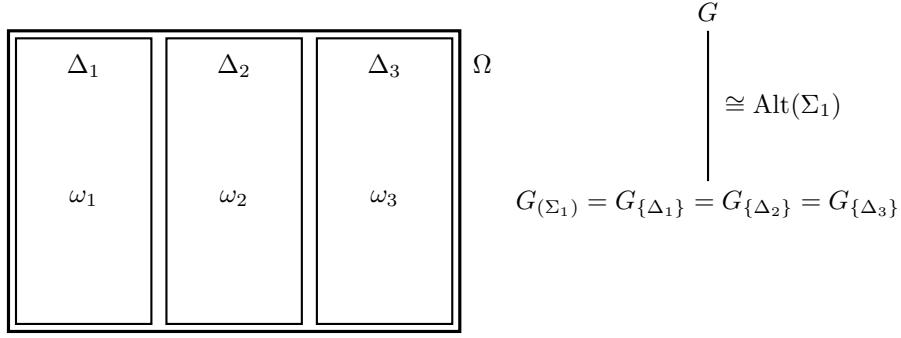
\begin{figure}
\begin{tikzpicture}
\draw[very thick] (-3,-2) rectangle (3,2);
\draw[thick] (-2.9,-1.9)  rectangle (-1.1,1.9);
\draw[thick] (-0.9,-1.9)  rectangle (.9,1.9);
\draw[thick] (1.1,-1.9)  rectangle (2.9,1.9);
\node[below] at (-2,0){$\omega_1$};
\node[below] at (0,0){$\omega_2$};
\node[below] at (2,0){$\omega_3$};
\node [below] at (-2,1.8){$\Delta_1$};
\node [below] at (0,1.8){$\Delta_2$};
\node [below] at (2,1.8){$\Delta_3$};
\node [below] at (3.3,1.8){$\Omega$};
\node[above] at (6.3,2){$G$};
\node[below] at (6.3,0){$G_{(\Sigma_1)}=G_{\{\Delta_1\}}=G_{\{\Delta_2\}}=G_{\{\Delta_3\}}$};
\node at (7.3,1){$\cong\mathrm{Alt}(\Sigma_1)$};
\draw[thick] (6.3,2)--(6.3,0);
\end{tikzpicture}
\caption{System of imprimitivity $\Sigma_1$.}\label{figure1}
\end{figure}
\begin{notation}\label{notationstep1}
{\rm 
We establish some necessary notation for the subsequent steps of the proof of Theorem~\ref{thrm:main}. From Lemma~\ref{l:step1}, we know that $G_{(\Sigma_1)}$ has no derangements in its action on $\Omega$, thus
\begin{equation}\label{eq:mainequation}
G_{(\Sigma_1)}=\bigcup_{\omega\in \Omega}G_\omega.
\end{equation}

Let
\begin{equation*}\Sigma_1=\{\Delta_1,\Delta_2,\Delta_3\}
\end{equation*}
and choose
\begin{equation*}
\omega_i\in \Delta_i,\, \text{for all }i\in \{1,2,3\}.
\end{equation*}

Since $G$ is imprimitive on $\Omega$ with a system of imprimitivity $\Sigma_1$, and $G$ induces $\mathrm{Alt}(\Sigma_1)$ on $\Sigma_1$ by Lemma~\ref{l:step1}, we embed
$$G\le H\mathrm{wr}\mathrm{Alt}(3),$$
where $H\le \mathrm{Sym}(\Delta)$, $H$ is transitive and $\Delta$ is a finite set with $|\Delta|=|\Delta_i|$. Under this identification, we may write
$$\Delta_1=\Delta\times\{1\},\, \Delta_2=\Delta\times\{2\},\, \Delta_3=\Delta\times\{3\}$$
and regard $G$ as endowed with its imprimitive action on $\Omega=\Delta\times\{1,2,3\}$. Specifically,
$$G_{(\Sigma_1)}\le H\times H\times H.$$
For each $i\in \{1,2,3\}$, let $\pi_i:G_{(\Sigma_1)}\to H$ be the projection onto the $i^{\mathrm{th}}$ direct factor. Without loss of generality (by potentially replacing $H$ with the permutation group induced by $G_{(\Sigma_1)}$ on its action on $\Delta_i$), we may assume that 
$$H=\pi_1(G_{(\Delta_1)})=\pi_2(G_{(\Sigma_1)})=\pi_3(G_{(\Sigma_1)}).$$
This allows a concrete representation of the elements of $G$; indeed, each element of $G$ takes the form $(h_1,h_2,h_3)(1\,2\,3)^i$, for some $h_1,h_2,h_3\in H$ and $i\in \{0,1,2\}$.

Let $c\in G\setminus G_{(\Sigma_1)}$ with $\Delta_1^c=\Delta_2$. Then $c=(h_1,h_2,h_3)(1\,2\,3)$, for some $h_1,h_2,h_3\in H$, and $G=G_{(\Sigma_1)}\langle c\rangle$. Let $h=h_1h_2h_3$. Note that
\begin{align*}
(1,h_1^{-1},h_2^{-1}h_1^{-1})^{-1}c(1,h_1^{-1},h_2^{-1}h_1^{-1})&=
(1,h_1,h_1h_2)(h_1,h_2,h_3)(1\,2\,3)(1,h_1^{-1},h_2^{-1}h_1)\\
&=(h_1,h_1h_2,h_1h_2h_3)(h_1^{-1},h_2^{-1}h_1^{-1},1)(1\,2\,3)\\
&=(1,1,h)(1\,2\,3).
\end{align*}
This shows that, by replacing $G$ with its conjugate $G^{(1,h_1^{-1},h_2^{-1}h_1^{-1})}$ in $H\mathrm{wr}\mathrm{Alt}(3)$, we may assume that
\begin{equation}\label{eq:c}
c=(1,1,h)(1\,2\,3).
\end{equation} Of course, this substitution does not affect the fact that $\Gamma_{G,\Omega}$ has no $4$-clique.

Notice that if $G$ is primitive on $\Omega$, then $\Sigma_1=\Omega$, implying that part~\eqref{thrm:maineq1} of Theorem~\ref{thrm:main} holds with $|\Omega|=3$. Hence, for the remainder of the proof of Theorem~\ref{thrm:main}, we assume that $G$ is imprimitive on $\Omega$, meaning $|\Omega|>3$.
}
\end{notation}

\section{Step 2}\label{sec:step2}In this section, we adopt Notation~\ref{notationstep1}. The objective of this section is to establish partial results towards the proof of Theorem~\ref{thrm:main}.
\begin{lemma}\label{l:step2}
The group $G$ admits a system of imprimitivity $\Sigma_2$, which is a refinement of $\Sigma_1$, such that the permutation group induced by $G$ on $\Sigma_2$ is permutation isomorphic to the action of $\mathrm{Alt}(4)$  on the six $2$-subsets of $\{1,2,3,4\}$. See Figure~$\ref{figure3}$ for some information on the subgroup lattice of $G$.
\end{lemma}
\begin{proof}
From Notation~\ref{notationstep1}, we have $\bar G=\mathrm{Alt}(\Sigma_1)$ and hence $G_{\{\Delta_i\}}=G_{(\Sigma_1)}$, because $\mathrm{Alt}(\Sigma_1)$ acts regularly on $\Sigma_1$. As $G_{\{\Delta_i\}}$ is transitive on $\Delta_i$, we deduce that 
\begin{align*}\label{eq:transitive}
G_{(\Sigma_1)}\hbox{ is transitive on }\Delta_i,\, \hbox{ for each }i\in \{1,2,3\}.
\end{align*}

Since $|\Omega|>|\Sigma_1|=3$, $G$ admits a system of imprimitivity $\Sigma_2\ne\Sigma_1$ which is a refinement of $\Sigma_1$.\footnote{Observe that, when $G_{\{\Delta_1\}}$ is primitive on $\Delta_1$, $\Sigma_2$ is the trivial system of imprimitivity consisting of singletons of $\Omega$.} Among all such systems of imprimitivity choose one so that $|\Sigma_2|$ is as small as possible. Just as in the proof of Lemma~\ref{l:step1}, we may replace $G$ with the permutation group induced by $G$ on $\Sigma_2$ and hence we may suppose that $\Sigma_2=\{\{\omega\}\mid\omega\in \Omega\}$. Therefore, we need to show that the action of $G$ on $\Omega$ itself is permutation isomorphic to the action of $\mathrm{Alt}(4)$  on the six $2$-subsets of $\{1,2,3,4\}$.

The minimality of $|\Sigma_2|$ implies that $G_{\{\Delta_1\}}=G_{(\Sigma_1)}$ is primitive in its action on $\Delta_1$. If $G_{(\Delta_1)}= 1$, then the hypotheses of Proposition~\ref{prop:technique1} are satisfied and hence $G_{(\Sigma_1)}$ contains a derangement in its action on $\Omega$, contradicting~\eqref{eq:mainequation} (or the conclusion of Lemma~\ref{l:step1}). This contradiction implies that 
\begin{equation*}
G_{(\Delta_1)}\ne 1,
\end{equation*} that is, $G_{\{\Delta_1\}}$ does not act faithfully on $\Delta_1$. Although we are in a position to apply Proposition~\ref{prop:technique2}, we choose not to do so here. Instead, we argue using an ad-hoc method, as we believe this approach could be applied more generally. However, at present, we do not know how to achieve this broader generalization.

Since $G_{(\Sigma_1)}$ is primitive on $\Delta_i$, we deduce that $G_{\omega_i}$ is a maximal subgroup of $G_{(\Sigma_1)}$. 

We claim that
\begin{equation}\label{eq:G}
G_{(\Delta_1\cup\Delta_2)}=G_{(\Delta_1\cup \Delta_3)}=G_{(\Delta_2\cup\Delta_3)}=1. 
\footnote{Observe that, since $G$ acts transitively by conjugation on $\Sigma_1=\{\Delta_1,\Delta_2,\Delta_3\}$, the groups 
$G_{(\Delta_1\cup \Delta_2)}$, $G_{(\Delta_1\cup \Delta_3)}$ and $G_{(\Delta_2\cup \Delta_3)}$ are $G$-conjugate. Hence, if one of them is trivial, all must be trivial.
}
\end{equation}
We argue by contradiction and we suppose that $G_{(\Delta_1\cup\Delta_2)}\ne 1$. Observe that 
$$G_{(\Delta_1)}\cap G_{(\Delta_2)}\cap G_{(\Delta_3)}=G_{(\Delta_1\cup\Delta_2\cup\Delta_3)}=G_{(\Omega)}=1.$$
Since $G_{(\Delta_1\cup \Delta_2)}\unlhd G_{(\Sigma_1)}$ and since $G_{\omega_3}$ is a maximal subgroup of $G_{(\Sigma_1)}$, we deduce that either $G_{(\Sigma_1)}=G_{\omega_3}G_{(\Delta_1\cup\Delta_2)}$ or $G_{(\Delta_1\cup\Delta_2)}\le G_{\omega_3}$. If $G_{(\Delta_1\cup\Delta_2)}\le G_{\omega_3}$, then, using the fact that $G_{(\Delta_1\cup\Delta_2)}\unlhd G_{(\Sigma_1)}$ and that $G_{(\Sigma_1)}$ is transitive on $\Delta_3$, we get 
$$G_{(\Delta_1\cup\Delta_2)}\le \bigcap_{k\in G_{(\Sigma_1)}}G_{\omega_3}^k=\bigcap_{\omega\in \Delta_3}G_{\omega}=G_{(\Delta_3)}.$$
Thus $G_{(\Delta_1\cup\Delta_2)}=G_{(\Delta_1\cup\Delta_2\cup\Delta_3)}=1$, which contradicts the fact that we are assuming $G_{(\Delta_1\cup\Delta_2)}\ne 1$.  Therefore, $$G_{(\Sigma_1)}=G_{\omega_3}G_{(\Delta_1\cup\Delta_2)}.$$
Since $G_{(\Sigma_1)}$ is transitive on $\Delta_3$, the Frattini argument implies that $G_{(\Delta_1\cup\Delta_2)}$ is transitive on $\Delta_3$. In particular, from Jordan's theorem, there exists $k_3\in G_{(\Delta_1\cup\Delta_2)}$ such that $k_3$ is a derangement on $\Delta_3$ and fixes each point in $\Delta_1\cup\Delta_2$. With a similar argument, we obtain $k_1\in G_{(\Delta_2\cup\Delta_3)}$ acting as a derangement on $\Delta_1$ and $k_2\in G_{(\Delta_1\cup\Delta_3)}$ acting as a derangement on $\Delta_2$. The support of the permutation $k_i$ is exactly $\Delta_i$ and hence $k=k_1k_2k_3\in G_{(\Sigma_1)}$ is a permutation having support exactly $\Delta_1\cup\Delta_2\cup\Delta_3=\Omega$. Therefore, $k$ is a derangement of $\Omega$ belonging to $G_{(\Sigma_1)}$, which contradicts~\eqref{eq:mainequation} (or the conclusion of Lemma~\ref{l:step1}). This contradiction has established the veracity of~\eqref{eq:G}.

Now, we use the embedding $G\le H\mathrm{wr}\mathrm{Alt}(3)$, see Notation~\ref{notationstep1}. 
Let 
\begin{align*}
A&:=\pi_2(G_{(\Delta_1)})\le \mathrm{Sym}(\Delta), \\
B&:=\pi_3(G_{(\Delta_1)})\le \mathrm{Sym}(\Delta).
\end{align*} Thus 
\begin{equation}\label{type1}
G_{(\Delta_1)}\le 1\times A\times B.
\end{equation}
Using the maximality of $G_{\omega_2}$ and $G_{\omega_3}$ in $G_{(\Sigma_1)}$, we get $G_{(\Sigma_1)}=G_{\omega_2}G_{(\Delta_1)}=G_{\omega_3}G_{(\Delta_1)}$.  Thus, from the Frattini argument, we deduce that $G_{(\Delta_1)}$ is transitive on  $\Delta_2$ and on $\Delta_3$, that is,
 $A$ and $B$ are transitive subgroups of $\Sym(\Delta)$.

Conjugating $G_{(\Delta_1)}$ by $c$\footnote{Recall $c=(1,1,h)(1\,2\,3)$, see Notation~\ref{notationstep1}.}, we deduce
\begin{align*}
G_{(\Delta_2)}&=G_{(\Delta_1)}^c\le (1\times A\times B)^c=(1\times A\times B)^{(1,1,h)(1\,2\,3)}\\
&=(1\times A\times B^h)^{(1\,2\,3)}=B^h\times 1\times A.\nonumber
\end{align*}
In fact, this computation shows that
\begin{align}\label{type2}
\pi_1(G_{(\Delta_2)})&=B^h \hbox{ and } \pi_3(G_{(\Delta_2)})=A. 
\end{align}

As $G_{(\Delta_1)},G_{(\Delta_2)}\unlhd G_{(\Sigma_1)}$, we have $$[G_{(\Delta_1)},G_{(\Delta_2)}]\le G_{(\Delta_1)}\cap G_{(\Delta_2)}=G_{(\Delta_1\cup\Delta_2)}=1,$$
where in the last equality we have used~\eqref{eq:G}. Hence $G_{(\Delta_1)}$ commutes with $G_{(\Delta_2)}$. From~\eqref{type1} and~\eqref{type2}, we deduce that $[A,B]=1$. As $A$ and $B$ are both transitive subgroups of $\mathrm{Sym}(\Delta)$, from Lemma~\ref{centralizes} we get $B={\bf C}_{\mathrm{Sym}(\Delta)}(A)$ and $A={\bf C}_{\mathrm{Sym}(\Delta)}(B)$.\footnote{We have $A=B$ if and only if $A$ is abelian and, in this case, $A$ acts regularly on $\Delta$. If $A\ne B$, then $A$ and $B$ both act regularly on $\Delta$ and we may identify the action of $A$ on $\Delta$ via its right regular representation and the action of $B$ on $\Delta$ via the left regular representation of $A$ itself.} 

Conjugating by $c$ again, we deduce
\begin{equation*}G_{(\Delta_3)}=G_{(\Delta_2)}^c\le (B^h\times 1\times A)^c=A^h\times B^h\times 1.
\end{equation*}
As above, this computation shows that
\begin{align}\label{type3}
\pi_1(G_{(\Delta_3)})&=A^h \hbox{ and } \pi_2(G_{(\Delta_2)})=B^h. 
\end{align}

Arguing as above, we see that $G_{(\Delta_3)}$ commutes with $G_{(\Delta_1)}$ and with $G_{(\Delta_2)}$. From~\eqref{type1} and~\eqref{type3}, we deduce that $[A,B^h]=1$. As $A$ and $B^h$ are both transitive subgroups of $\mathrm{Sym}(\Delta)$, from Lemma~\ref{centralizes} we get $$B^h={\bf C}_{\mathrm{Sym}(\Delta)}(A)=B.$$ Thus $h$ normalizes $B$. As $A={\bf C}_{\mathrm{Sym}(\Delta)}(B)$, we get $$A^h={\bf C}_{\mathrm{Sym}(\Delta)}(B^h)={\bf C}_{\mathrm{Sym}(\Delta)}(B)=A.$$ Thus $h$ normalizes $A$. Now, from~\eqref{type1},~\eqref{type2} and~\eqref{type3}, we have
\begin{align}\label{combined}
G_{(\Delta_1)}&\le 1\times A\times B,\\\nonumber
G_{(\Delta_2)}&\le B\times 1\times A,\\\nonumber
G_{(\Delta_3)}&\le A\times B\times 1,\nonumber
\end{align}
where $B={\bf C}_{\Sym(\Delta)}(A)$ and $A$ is a regular subgroup of $\mathrm{Sym}(\Delta)$.

Now, we use~\eqref{eq:G} combined with~\eqref{combined}. Indeed, from~\eqref{eq:G} and from the first line in~\eqref{combined}, we deduce that, for every $a\in A$, there exists a unique $b\in B$ with $(1,a,b)\in G_{(\Delta_1)}$. In particular, the mapping $\varphi: A \to B$, characterized by the property that $(1, a, a^\varphi) \in G_{(\Delta_1)}$ for every $a \in A$, defines a group isomorphism from $A$ to $B$.
Similarly,~\eqref{eq:G} and the remaining two lines of~\eqref{combined} imply that there exist two more group isomorphisms $\psi,\eta:A\to B$ with
\begin{align*}
G_{(\Delta_1)}&=\{(1,a,a^\varphi)\mid a\in A\},\\
G_{(\Delta_2)}&=\{(a^\psi,1,a)\mid a\in A\},\\
G_{(\Delta_3)}&=\{(a,a^\eta,1)\mid a\in A\}.
\end{align*}

We are now ready to conclude the proof of this result. Since $A$ and $B$ are commuting regular subgroups of $\operatorname{Sym}(\Delta)$, we may apply Lemma~\ref{centralizers2}. Assume there exist $a \in A \setminus \{1\}$ and $b \in B \setminus \{1\}$ such that $ab$ is a derangement (on $\Delta$). Then
$$
\underbrace{(a^\psi,1,a)}_{\in G_{(\Delta_2)}} \underbrace{(1,b^{\varphi^{-1}},b)}_{\in G_{(\Delta_1)}} = (a^\psi, b^{\varphi^{-1}}, ab) \in G_{(\Sigma_1)}.
$$
The first two coordinates of this element lie in regular subgroups and are different from the identity since $\varphi$ and $\psi$ are group isomorphisms. Therefore, each coordinate acts as a derangement on its corresponding block, and we deduce that $(a^\psi, b^{\varphi^{-1}}, ab)$ is a derangement on $\Omega$, which contradicts~\eqref{eq:mainequation}. This contradiction shows that only the other case of Lemma~\ref{centralizers2} can occur, namely $|\Delta| = 2$ and $|\Omega| = |\Sigma_1||\Delta| = 6$.\footnote{The fact that $G$ is permutation-isomorphic to $\operatorname{Alt}(4)$ in its action on the six 2-subsets of $\{1,2,3,4\}$ follows easily.}
\end{proof}
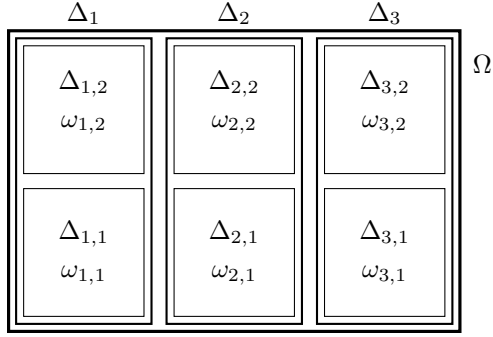
\begin{figure}
\begin{tikzpicture}
\draw[very thick] (-3,-2) rectangle (3,2);
\draw[thick] (-2.9,-1.9)  rectangle (-1.1,1.9);
\draw[thick] (-0.9,-1.9)  rectangle (.9,1.9);
\draw[thick] (1.1,-1.9)  rectangle (2.9,1.9);
\node[below] at (-2,-1){$\omega_{1,1}$};
\node[below] at (0,-1){$\omega_{2,1}$};
\node[below] at (2,-1){$\omega_{3,1}$};
\node[below] at (-2,1){$\omega_{1,2}$};
\node[below] at (0,1){$\omega_{2,2}$};
\node[below] at (2,1){$\omega_{3,2}$};

\node[above] at (2,1){$\Delta_{3,2}$};
\node[above] at (0,1){$\Delta_{2,2}$};
\node[above] at (-2,1){$\Delta_{1,2}$};

\node[above] at (2,-1){$\Delta_{3,1}$};
\node[above] at (0,-1){$\Delta_{2,1}$};
\node[above] at (-2,-1){$\Delta_{1,1}$};

\node [below] at (-2,2.5){$\Delta_1$};
\node [below] at (0,2.5){$\Delta_2$};
\node [below] at (2,2.5){$\Delta_3$};
\node [below] at (3.3,1.8){$\Omega$};
\draw (-2.8,-1.8) rectangle (-1.2,-.1);
\draw (-2.8,.1) rectangle (-1.2,1.8);

\draw (-.8,-1.8) rectangle (.8,-.1);
\draw (-.8,.1) rectangle (.8,1.8);

\draw (1.2,-1.8) rectangle (2.8,-.1);
\draw (1.2,.1) rectangle (2.8,1.8);
\end{tikzpicture}
\caption{Systems of imprimitivity $\Sigma_1$ and $\Sigma_2$.}\label{figure2}
\end{figure}
\begin{figure}
\begin{tikzpicture}
\node[above] at (0,2){$G$};
\node at (0,1){$G_{(\Sigma_1)}=G_{\{\Delta_1\}}=G_{\{\Delta_2\}}=G_{\{\Delta_3\}}$};
\node[right] at (0,1.5){$\cong\mathrm{Alt}(\Sigma_1)$};
\draw[thick] (0,2)--(0,1.2);
\node at (0,-1){$G_{\{\Delta_{2,1}\}}=G_{\{\Delta_{2,2}\}}$};
\draw[thick](0,.8)--(0,-.8);
\node at (-3,-1){$G_{\{\Delta_{1,1}\}}=G_{\{\Delta_{1,2}\}}$};
\draw[thick](0,.8)--(-3,-.8);
\node at (3,-1){$G_{\{\Delta_{3,1}\}}=G_{\{\Delta_{3,2}\}}$};
\draw[thick](0,.8)--(3,-.8);
\node at (0,-3){$G_{(\Sigma_2)}$};
\draw[thick](0,-1.2)--(0,-2.8);
\draw[thick](3,-1.2)--(0,-2.8);
\draw[thick](-3,-1.2)--(0,-2.8);
\draw    (4,2) to[out=-45,in=45] (4,-3);
\node at (6,-.5){$\cong\mathrm{Alt(4)}$};
\end{tikzpicture}
\caption{Some subgroups of $G$}\label{figure3}
\end{figure}
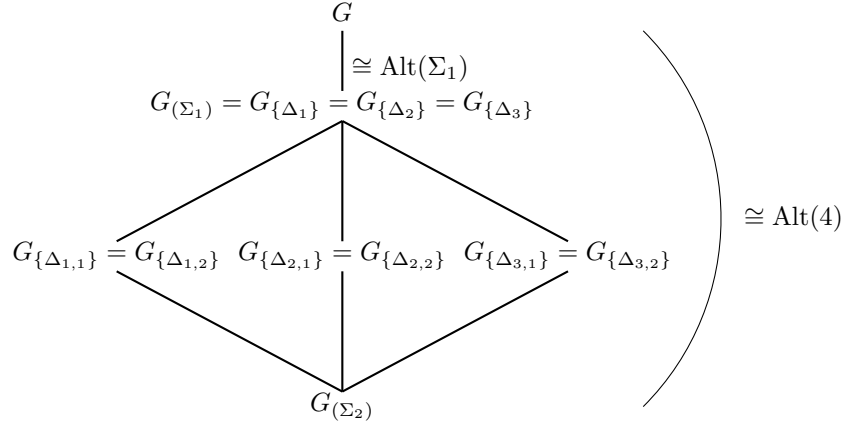

\begin{notation}\label{notationstep2}
{\rm Building upon Notation~\ref{notationstep1} and utilizing Lemma~\ref{l:step2}, we introduce additional notation crucial for the subsequent steps of the proof.

 We let $\Sigma_2$ be a system of imprimitivity of $G$ on $\Omega$ having cardinality $6$, which is a refinement of $\Sigma_1$. We let
$$\Sigma_2=\{\Delta_{1,1},\Delta_{1,2},\Delta_{2,1},\Delta_{2,2},\Delta_{3,1},\Delta_{3,2}\},$$
where $$\Delta_i=\Delta_{i,1}\cup\Delta_{i,2},\,\hbox{ for all }i\in \{1,2,3\}.$$
Moreover, we choose $\omega_{i,j}\in \Delta_{i,j}$, for all $i\in \{1,2,3\}$ and $j\in \{1,2\}$.

Using the system of imprimitivity $\Sigma_2$, we obtain an embedding
$$G\le K\mathrm{wr}\mathrm{Alt}(4),$$
where $K\le\mathrm{Sym}(\Delta)$, $\Delta$ is a finite set\footnote{In Notation~\ref{notationstep1}, the symbol $\Delta$ had a different meaning. However, since we no longer utilize the set $\Delta$ as defined in Notation~\ref{notationstep1}, we re-purpose the symbol $\Delta$ here to avoid introducing additional cumbersome notation.} with $|\Delta|=|\Delta_{i,j}|$ and $K$ is the permutation group induced by the action of $G_{\{\Delta_{i,j}\}}$ on $\Delta_{i,j}$. Under this identification, we may write 
\begin{align*}
\Delta_{1,1}&=\Delta\times\{1\}, \Delta_{1,2}=\Delta\times\{2\},\\
 \Delta_{2,1}&=\Delta\times\{3\}, \Delta_{2,2}=\Delta\times\{4\},\\ 
\Delta_{3,1}&=\Delta\times\{5\}, \Delta_{3,2}=\Delta\times\{6\}
\end{align*}
 and we may think about $G$ as endowed of its imprimitive action on $\Omega=\Delta\times\{1,2,3,4,5,6\}$. In particular, 
$$G_{(\Sigma_2)}\le K\times K\times K\times K\times K\times K.$$

For each $i\in \{1,2,3\}$ and $j\in \{1,2\}$, let $\pi_{i,j}:G_{(\Sigma_2)}\to K$ be the projection obtained by restricting a permutation of $G_{(\Sigma_2)}$ to $\Delta_{i,j}$. Since the union of the $G$-conjugates of $G_{\omega_{1,1}}$ is $G_{(\Sigma_1)}$ by~\eqref{eq:mainequation}, we have $G_{\omega_{1,1}}\nleq G_{(\Sigma_2)}$. As $[G_{\{\Delta_{1,1}\}}:G_{(\Sigma_2)}]=2$ and $G_{\omega_{1,1}}\le G_{\{\Delta_{1,1}\}}$, we deduce
\begin{equation*}
G_{\{\Delta_{1,1}\}}=G_{\omega_{1,1}}G_{(\Sigma_2)}.
\end{equation*}
Now, the Frattini argument implies that $G_{(\Sigma_2)}$ is transitive on $\Delta_{1,1}$. As $G_{(\Sigma_2)}\unlhd G$, we get that $G_{(\Sigma_2)}$ is transitive on $\Delta_{i,j}$, for every $i\in 
\{1,2,3\}$ and $j\in \{1,2\}$.

We let
 $$L=\pi_{i,j}(G_{(\Sigma_2)}),$$
 observe that $L$ does not depend on $i,j$ and, from the paragraph above, $L$ is a transitive subgroup of $K$ with $[K:L]\le 2$.\footnote{We have $[K:L]=2$ when $G_{(\Delta_{1,1})}\le G_{(\Sigma_2)}$, whereas $[K:L]=1$ when 
 $G_{(\Delta_{1,1})}\nleq G_{(\Sigma_2)}$.}

As $[G_{\{\Delta_i\}}:G_{\{\Delta_{i,j}\}}]=2$, we get\footnote{Some thought is actually needed to see this. Indeed, as $[G_{\{\Delta_i\}}:G_{\{\Delta_{i,j}\}}]=2$, we have either $G_{\{\Delta_i\}}=G_{\{\Delta_{i,j}\}}G_\omega$ or $G_\omega\le G_{\{\Delta_{i,j}\}}$. In the latter case, if we let $\Delta_{i',j'}$ be the block of $\Sigma_2$ containing $\omega$, we have $G_\omega\le G_{\{\Delta_{i,j}\}}\cap G_{\{\Delta_{i',j'}\}}=G_{(\Sigma_2)}$, see Figure~\ref{figure2}. However, this contradicts the fact that $G_{(\Sigma_2)}$ is transitive on each block of $\Sigma_2$.}
\begin{equation}\label{eq:stabilizers}
G_{\{\Delta_i\}}=G_{\{\Delta_{i,j}\}}G_\omega,\,\forall i\in \{1,2,3\}, j\in \{1,2\}, \omega\in \Omega\setminus\Delta_i.
\end{equation}

The embedding of $G$ in $K\mathrm{wr}\mathrm{Alt}(4)$ allows a concrete representation of the elements of $G$; indeed, each element of $G$ is of the form $(k_1,k_2,k_3,k_4,k_5,k_6)\sigma$, for some $k_1,k_2,k_3,k_4,k_5,k_6\in K$ and $$\sigma\in \langle (1\,3\,5)(2\,4\,6),(1\,2)(3\,4)\rangle\cong\mathrm{Alt}(4).$$ 
In particular, the element $c$ defined in Notation~\ref{notationstep1} can now be written as 
\begin{equation}\label{newc}
c=(1,1,1,1,k_1,k_2)(1\,3\,5)(2\,4\,6).\footnote{There is a slight abuse here. Indeed, the element $c=(1,1,h)(1\,2\,3)\in H\mathrm{wr}\mathrm{Alt}(3)$ viewed as an element of $K\mathrm{wr}\mathrm{Alt}(4)$ is $(1,1,1,1,k_1,k_2)(1\,3\,5)(2\,4\,6)$, for some $k_1,k_2\in K$.}
\end{equation}

Observe that, 
if $G_{(\Sigma_1)}$ is primitive on $\Delta_i$, then $\Sigma_2=\Omega$ and hence we deduce that part~\eqref{thrm:maineq2} of Theorem~\ref{thrm:main} is satisfied.\footnote{Strictly speaking,
to establish part~\eqref{thrm:maineq2} of Theorem~\ref{thrm:main}, it is necessary to omit the assumption in~\eqref{assumption1}. This is deduced through computation: the derangement graph of every proper overgroup of $\mathrm{Alt}(4)$ in its degree 6 action contains a clique of size $4$.
} Therefore, in the rest of the proof of Theorem~\ref{thrm:main} we may assume that $G_{(\Sigma_1)}$ is imprimitive on $\Delta_i$, that is, $|\Omega|>6=|\Sigma_2|$.

It is quite remarkable that Lemma~\ref{l:step2} has already shown that $|\Omega|$ is a multiple of $6$.}
\end{notation}

\section{Step 3}\label{sec:step3}
In this section, we utilize Notation~\ref{notationstep2}. The objective of this section is to establish a partial result towards proving Theorem~\ref{thrm:main}. 
\begin{lemma}\label{l:step3}
The group $G$ admits a system of imprimitivity $\Sigma_3$, which is a refinement of $\Sigma_2$, such that the permutation group induced by $G$ on $\Sigma_3$ is permutation equivalent to the group $(18,142)$, $(30,126)$ or $(30,233)$ in the library of transitive groups in \texttt{magma}.
\end{lemma}
\begin{proof}
Since $|\Omega| > |\Sigma_2 | = 6$, $G$ admits a system of imprimitivity $\Sigma_3\ne \Sigma_2$ which is a
refinement of $\Sigma_2$. Among all such systems of imprimitivity choose one so that $|\Sigma_3 |$
is as small as possible. Just as in the proof of Lemma~\ref{l:step1} and of Lemma~\ref{l:step2}, we may replace $G$ with the permutation group induced by $G$ on $\Sigma_3$ and hence
we may suppose that $\Sigma_3 = \{\{\omega\} \mid \omega \in \Omega\}$.

If $G_{(\Sigma_2)}=1$, then $G$ acts faithfully on $\Sigma_2$ and hence $G\cong \mathrm{Alt}(4)$, because the permutation group induced by $G$ on $\Sigma_2$ is $\mathrm{Alt}(4)$ in its degree six action. Since $12=|G|\ge |\Omega|>|\Sigma_2|=6$, we deduce $|\Omega|=12$ and $G$ acts regularly on $\Omega$, which is clearly a contradiction. Therefore 
\begin{equation}\label{todayistuesday}G_{(\Sigma_2)}\ne 1.
\end{equation}

The minimality of $|\Sigma_3|$ implies that $G_{\{\Delta_{1,1}\}}$ is primitive in its action
on $\Delta_{1,1}$. If $G_{(\Delta_{1,1})}= 1$, then the hypothesis of Proposition~\ref{prop:technique1} are satisfied and
hence $G_{(\Sigma_2)}$ contains a derangement in its action on $\Omega$, contradicting~\eqref{eq:mainequation}. This
contradiction implies that
\begin{equation}\label{foggyday}
G_{(\Delta_{i,j})}\ne 1,\,\forall i\in \{1,2,3\}, j\in \{1,2\}.
\end{equation}

Suppose that, the primitive permutation group induced by $G_{\{\Delta_{i,j}\}}$ on $\Delta_{i,j}$ has non-abelian socle. We now apply Proposition~\ref{prop:technique2}  and the notation therein. Since by~\eqref{eq:mainequation}, $G_{(\Sigma_2)}$ has no derangements, we deduce that parts~\eqref{technique2:nr1}--\eqref{technique2:nr5} of Proposition~\ref{prop:technique2} are satisfied. Recall that $G/G_{(\Sigma_2)}\cong\mathrm{Alt}(4)$ has degree $6$. It is elementary to check either by hand or with a computer that, if $\Gamma$ is an $(a,b)$-hypergraph such that $\mathrm{Alt}(4)$ is special for $\Gamma$ and such that the stabilizer of a point in $\mathrm{Alt}(4)$ (in its action of degree $6$) fixes some edge of $\Gamma$, then $\chi(\Gamma)\le 4$.

Observe that, by Burnside $p^\alpha q^\beta$ theorem, a non-abelian simple group has order divisible by at least three distinct primes and hence the non-abelian simple group $T$ appearing in the statement of Proposition~\ref{prop:technique2} has at least $4$ $\mathrm{Aut}(T)$-conjugacy classes. From Proposition~\ref{prop:technique2} part~\eqref{technique2:nr42}, Lemma~\ref{l:conjclasses} and the previous paragraph, we immediately obtain a contradiction.

Assume that the primitive permutation group induced by $G_{\{\Delta\}}$ on $\Delta$ has abelian
socle. Let $V$ be a minimal normal subgroup of $G$ with $V\le G_{(\Sigma_2)}$. Then, $V$ is an elementary abelian $p$-group, for some prime number $p$. From~\eqref{eq:mainequation}, we have 
$$V=\bigcup_{\omega\in \Omega}V_\omega.$$
Since $V_\omega$ is normalized by $V$ and by $G_\omega$, we get ${\bf N}_G(V_\omega)\ge G_\omega V=G_{\{\Delta\}}$ and hence the union above consists of at most $|\Sigma_3|$ elements. Therefore $$p^a=|\Delta|=|V:V_{\omega}|<|\Sigma_2|=6.$$
Thus $|\Omega|=p^a|\Sigma_2|\le 30$. The rest of the result follows with a computer computation using the database of transitive groups of degree at most $30$. 
\end{proof}

\section{Final step}
The objective of this section  is to establish  Theorem~\ref{thrm:main}.

\begin{lemma}\label{computer1}Let $G$ be the group $(18,142)$ or $(30,126)$ in the library of transitive groups in \texttt{magma} and let $\Gamma$ be an $(a,b)$-hypergraph $\Gamma$ such that $G$ is special for $\Gamma$ and such that the stabilizer of a point in $G$ fixes some edge of $\Gamma$.  Then $\chi(\Gamma)\le 4$.
\end{lemma}
\begin{proof}
This follows with a computer computation. The computation is elementary. Given $G$ as above, we have constructed all $(a,b)$-hypergraphs $\Gamma$ such that $G$ is special for $\Gamma$. Then, for each $(a,b)$-hypergraph $\Gamma=(V,E)$ we have randomly constructed $10^6$ maps $\eta:V\to \{1,2,3,4\}$ and we have checked whether this map is a colouring of $\Gamma$. In all cases, we found such a map. Incidentally, these $(a,b)$-hypergraphs might have a smaller chromatic number.
\end{proof}
\begin{theorem}[{{[See, Theorems~1 and~2 in~\cite{bereczky}]}}]\label{bere}
Let $p$ be an odd prime number and $a\ge 1$. If $p+1<b<3(p+1)/2$ or if $b=p+1$ and $a\ge 2$, then every transitive permutation group of degree $p^a\cdot b$ contains a derangement of $p$-power order.
\end{theorem}

\begin{proof}[Proof Theorem~$\ref{thrm:main}$]
We use the same notation that we have established throughout Sections~\ref{sec:step1},~\ref{sec:step2} and~\ref{sec:step3}. 

Let $\Sigma_4$ be a system of imprimitivity for $G$, which is a refinement of $\Sigma_3$. Just as in the proof of Lemmas~\ref{l:step1},~\ref{l:step2} and~\ref{l:step3}, we may replace $G$ with the permutation group induced by $G$ on $\Sigma_4$ and hence we may suppose that $\Sigma_4 =\{\{\omega\} | \omega \in \Omega\}$.

Suppose first $G_{(\Sigma_3)}=1$. Then $G$, as an abstract group, is isomorphic to the permutation group induced by $G$ on $\Sigma_3$ and hence it is isomorphic to the transitive group $(18,142)$ or $(30,126)$ in the library of transitive groups in \texttt{magma}~\cite{magma}. We have checked with a computer that these two permutation groups of degree $18$ and $30$ do not admit a transitive action on a set $|\Omega|>|\Sigma_3|$ such that the derangement graph has no $4$-cliques. Therefore $G_{(\Sigma_3)}\ne 1$. Let $\Delta\in \Sigma_3$.

Suppose $G_{(\Delta)}=1$.  Then, by Proposition~\ref{prop:technique1}, $G_{(\Sigma_3)}$ admits a derangement for its action on $\Omega$, contradicting~\eqref{eq:mainequation}. Therefore, $G_{(\Delta)}\ne 1$. 

Assume that the primitive permutation group induced by $G_{\{\Delta\}}$ on $\Delta$ has non-abelian
socle. We now apply Proposition~\ref{prop:technique2} and the notation therein. Since by~\eqref{eq:mainequation}, $G_{(\Sigma_3)}$ has no derangements, we deduce that parts~\eqref{technique2:nr1}--\eqref{technique2:nr5} of Proposition~\ref{prop:technique2} are satisfied. Observe that, by Burnside $p^\alpha q^\beta$ theorem, a non-abelian simple group has order divisible by at least three distinct primes and hence the non-abelian simple group $T$ appearing in the statement of Proposition~\ref{prop:technique2} has at least $4$ $\mathrm{Aut}(T)$-conjugacy classes. From Proposition~\ref{prop:technique2} part~\eqref{technique2:nr42}, Lemma~\ref{l:conjclasses} and Lemma~\ref{computer1}, we immediately obtain a contradiction.

Assume that the primitive permutation group induced by $G_{\{\Delta\}}$ on $\Delta$ has abelian
socle. Let $V$ be a minimal normal subgroup of $G$ with $V\le G_{(\Sigma_3)}$. Then, $V$ is an elementary abelian $p$-group, for some prime number $p$. From~\eqref{eq:mainequation}, we have 
$$V=\bigcup_{\omega\in \Omega}V_\omega.$$
Since $V_\omega$ is normalized by $V$ and by $G_\omega$, we get ${\bf N}_G(V_\omega)\ge G_\omega V=G_{\{\Delta\}}$ and hence the union above consists of at most $|\Sigma_3|$ elements. Therefore $$p^a=|\Delta|=|V:V_{\omega}|<|\Sigma_3|.$$ Observe that, since $|G:G_{(\Sigma_1)}|=3$ and $G_{(\Sigma_1)}\unlhd G$, if $p\ne 3$, then every $p$-element of $G$ is contained in $G_{(\Sigma_1)}$. Hence $G$ has no derangements of $p$-power order with $p\ne 3$ by~\eqref{eq:mainequation}. Therefore, using Theorem~\ref{bere}, we get
\begin{itemize}
\item $|\Sigma_3|=18$ and $|\Delta|\in \{2,3,4,5,7,8,9,11,16,17\}$, or
\item $|\Sigma_3|=30$ and $|\Delta|\in \{2,3,4,7,8,9,11,13,16,17,19,27,29\}$.
\end{itemize}
All of these cases have been tested with a computer and no counterexample arises. The main tool used in these computations is the built-in function \texttt{MaximalSubgroups}  in \texttt{magma}, with ad-hoc arguments for some of the cases.
\end{proof}

%\bibitem{wall}G.~E.~Wall, On the conjugacy classes in the unitary, symplectic and orthogonal groups, \textit{J. Aust. Math. Soc.} \textbf{3} (1963), 1--62.

%\bibitem{wielandt2014finite}
%H.~Wielandt, \textit{Finite {P}ermutation {G}roups}, Academic Press, New York, 1964.
%\bibitem{zantema}H.~Zantema, Integer valued polynomials over a number ﬁeld, \textbf{Manuscr. Math.} \textbf{40} (1982) 155--203.

%\bibitem{zi}K.~Zsigmondy, Zur Theorie der Potenzreste,
%\textit{Monatsh. Math. Phys.} \textbf{3} (1892), 265--284.

\end{document}